\documentclass[12pt, twoside, reqno]{amsart}
\usepackage{amscd}
\usepackage{bbm}
\usepackage{dsfont}
\usepackage{enumerate}
\usepackage{latexsym}
\usepackage{srcltx}
\usepackage{amsfonts}
\usepackage{amssymb}
\usepackage{datetime}
\usepackage{setspace}
\usepackage{enumerate}
\usepackage{amsthm}
\usepackage{graphicx}
\usepackage{epstopdf}

\newtheorem{theorem}{Theorem}[section]

\newtheorem{lemma}[theorem]{Lemma}
\newtheorem{corollary}[theorem]{Corollary}

\theoremstyle{definition}
\newtheorem{example}[theorem]{Example}

\newtheorem{remark} [theorem] {Remark}

\input epsf
\begin{document}

\title{ Spectrum of Weighted Composition Operators \\
Part VII \\
Essential spectra of weighted composition operators
on $C(K)$. The case of non-invertible homeomorphisms.}

\author{Arkady Kitover}

\address{Community College of Philadelphia, 1700 Spring Garden St., Philadelphia, PA,
USA}

\email{akitover@ccp.edu}

\author{Mehmet Orhon}

\address{University of New Hampshire, Durham, NH, 03824}

\email{mo@unh.edu}


\date{\today}


\begin{abstract}
  We provide a complete description of the spectrum and the essential spectra of weighted composition operators $T=wT_\varphi$ on $C(K)$ in the case when the map $\varphi$ is a non-invertible homeomorphism of $K$  into itself.
 \end{abstract}
\maketitle

\markboth{Arkady Kitover and Mehmet Orhon}{Spectrum of weighted composition operators.
VII}

\section{Introduction}
Let $K$ be a compact Hausdorff space and $C(K)$ be the space of all complex-valued continuous functions on $K$. A weighted composition operator $T$ on $C(K)$ is an operator of the form
\begin{equation*}
  (Tf)(k)= w(k)f(\varphi(k), k \in K, f \in C(K),
\end{equation*}
where $\varphi$ is a continuous map of $K$ into itself and $w \in C(K)$.

The spectrum of arbitrary weighted composition operators on $C(K)$ was investigated by the first named author in~\cite[Theorems 3.10, 3.12, and 3.23]{Ki1}. On the other hand, the full description of \textbf{essential} spectra (in particular Fredholm and semi-Fredholm spectra) of such operators is, as far as we are informed, still not known. In a special case, when the map $\varphi$ is a homeomorphism of $K$ \textbf{onto} itself, such a description was obtained in~\cite[Theorems 2.7 and 2.11]{Ki3}. In this paper we provide a description of the spectrum (Theorem~\ref{t1}) and the essential spectra (Theorems~\ref{t2},~\ref{t3}, and~\ref{t4}) of a weighted composition operator $T = wT_\varphi$ in the case when $\varphi$ is a \textbf{non-surjective} homeomorphism of $K$ into itself.

\section{Preliminaries}

In the sequel we use the following standard notations.

\noindent $\mathds{N}$ is the semigroup of all natural numbers.

\noindent $\mathds{Z}$ is the ring of all integers.

\noindent $\mathds{R}$ is the field of all real numbers.

\noindent $\mathds{C}$ is the field of all complex numbers.

\noindent $\mathds{T}$ is the unit circle. We use the same notation for the unit circle
considered as a subset of the complex plane and as the group of all complex numbers
of modulus 1.

\noindent $\mathds{U}$ is the open unit disc.

\noindent $\mathds{D}$ is the closed unit disc.

All the linear spaces are considered over the field $\mathds{C}$ of complex numbers.

The algebra of all bounded linear operators on a Banach space $X$ is denoted by $L(X)$.

Let $E$ be a set and $\varphi : E \rightarrow E$ be a map of $E$ into itself. Then

\noindent $\varphi^n$ , $n \in \mathds{N}$, is the $n^{th}$ iteration of $\varphi$,

\noindent $\varphi^0(e) = e, e \in E$,

\noindent If $F \subseteq E$ then $\varphi^{(-n)}(F)$ means the full $n^{th}$ preimage of $F$, i.e.
$\varphi^{(-n)}(F)= \{e \in E : \varphi^n(e) \in F\}$.

\noindent If the map $\varphi$ is imjective then $\varphi^{-n}$, $n \in \mathds{N}$,  is the $n^{th}$ iteration of the inverse map $\varphi^{-1}$. In this case we will write $\varphi^{-n}(F)$ instead of $\varphi^{(-n)}(F)$.

\noindent Let $w$ be a complex-valued function on $E$. Then $w_0 = 1$ and $w_n = w(w \circ \varphi) \ldots (w \circ \varphi^{n-1})$, $n \in \mathds{N}$.

Recall that an operator $T \in L(X)$ is called \textit{semi-Fredholm} if its range
$R(T)$ is closed in $X$ and either $\dim{\ker{T}}< \infty$ or codim $R(T) < \infty$.

The \textit{index} of a semi-Fredholm operator $T$ is defined as

\centerline{ ind $T$ = $  \dim{\ker{T}}$ - $\mathrm{codim} \, R(T)$.}

The subset of $L(X)$ consisting of all semi-Fredholm operators is denoted by $\Phi$.

$\Phi_+ = \{T \in \Phi : null(T) = \dim{\ker{T}}< \infty\}$ is the set of all upper semi-Fredholm operators in $L(X)$.

$\Phi_- = \{T \in \Phi : def(T) = \textrm{codim}\; R(T) < \infty\}$ is the set of all lower semi-Fredholm operators in $L(X)$.

$\mathcal{F} = \Phi_+  \cap  \Phi_-$ is the set of all Fredholm operators in $L(X)$.

$\mathcal{W} = \{T \in \mathcal{F} : \textrm{ind} \; T = 0\}$ is the set of all Weyl
operators in $L(X)$.

Let $T$ be a bounded linear operator on a Banach space $X$. As usual, we denote the
spectrum of $T$ by $\sigma(T)$ and its spectral radius by $\rho(T)$.

We will consider the following subsets of $\sigma(T)$.

$\sigma_p(T) = \{\lambda \in \mathds{C} : \exists x \in X \setminus \{0\}, Tx = \lambda
x\}.$

$\sigma_{a.p.}(T) = \{\lambda \in \mathds{C}: \exists x_n \in X, \|x_n\| = 1, Tx_n -
\lambda x_n \rightarrow 0\}$.

$\sigma_r(T) = \sigma(T) \setminus \sigma_{a.p.}(T) =$

$\; \; = \{\lambda \in \sigma(T) : \textrm{the operator}\; \lambda I - T \; \textrm{has the left inverse}\} $.

\begin{remark} \label{r1} It is clear that $\sigma_{a.p.}(T)$ is the union of the point spectrum $\sigma_p(T)$ and the approximate point spectrum $\sigma_a(T)$ of $T$, while $\sigma_r(T)$ is the residual spectrum of $T$. We have to notice that the  definition of the residual spectrum varies in the literature.
\end{remark}

\begin{remark} \label{r2} If needed to avoid an ambiguity, we will use notations $\sigma(T,X)$, $\rho(T,X)$, et cetera.
  \end{remark}

Following~\cite{EE} we consider the following essential spectra of $T$.

$\sigma_1(T) = \{\lambda \in \mathds{C}: \lambda I - T \not \in \Phi\}$ is the
\textit{semi-Fredholm} spectrum of $T$.

$\sigma_2(T) = \{\lambda \in \mathds{C}: \lambda I - T \not \in \Phi_+\}$ is the upper \textit{semi-Fredholm} spectrum of $T$.

$\sigma_2(T^\prime)  = \{\lambda \in \mathds{C}: \lambda I - T \not \in \Phi_-\} $ is the lower \textit{semi-Fredholm} spectrum of $T$.

$\sigma_3(T) = \{\lambda \in \mathds{C}: \lambda I - T \not \in \mathcal{F}\}$ is the
Fredholm spectrum of $T$.

$\sigma_4(T) = \{\lambda \in \mathds{C}: \lambda I - T \not \in \mathcal{W}\}$ is the
Weyl spectrum of $T$.

$\sigma_5(T) = \sigma(T)\setminus \{\zeta \in \mathds{C} :$ there is a component $C$ of
the set $\mathds{C} \setminus \sigma_1(T)$ such that $\zeta \in C$ and the intersection
of $C$ with the resolvent set of $T$ is not empty$\}$ is the Browder spectrum of $T$.

The Browder spectrum was introduced in~\cite{Br} as follows: $\lambda \in \sigma(T) \setminus \sigma_5(T)$ if and only if $\lambda$ is a pole of the resolvent $R(\lambda, T)$. It is not difficult to see (~\cite[p. 40]{EE}) that the definition of $\sigma_5(T)$ cited above is equivalent to the original definition of Browder.

It is well known (see e.g.~\cite{EE}) that the sets $\sigma_i(T ), i \in [1, \ldots ,
5]$ are
nonempty closed subsets of $\sigma(T)$ and that
\begin{equation*}
  \sigma_i(T) \subseteq \sigma_j(T), 1 \leq i < j \leq 5,
\end{equation*}
where all the inclusions can be proper. Nevertheless all the
spectral radii $\rho_i(T ), i = 1, . . . , 5$ are equal to the same number,
$\rho_e(T)$, (see~\cite[Theorem I.4.10]{EE}) which is called the essential
spectral radius of $T$. It is also known (see~\cite{EE}) that the spectra $\sigma_i(T),
i = 1, \ldots , 4$ are invariant under compact perturbations, but $\sigma_5(T)$ in
general is not.

It is immediate to see that $\sigma_1(T) = \sigma_2(T) \cap \sigma_2(T^\prime)$ and that $\sigma_3(T) = \sigma_2(T) \cup \sigma_2(T^\prime)$.\

Let us recall that a sequence $x_n$ of elements of a Banach space $X$ is called \emph{singular} if it does not contain any norm convergent subsequence. We will use the following well known characterization of $\sigma_2(T)$ (see e.g.~\cite{EE}). The following statements are equivalent
\begin{enumerate}[(a)]
  \item $\lambda \in \sigma_2(T)$.
  \item There is a singular sequence $x_n$ such that $\|x_n\| = 1$ and $\lambda x_n - Tx_n \rightarrow 0$.
\end{enumerate}

\section{The spectrum of $T = wT_\varphi$}

Let $K$ be a compact Hausdorff space, $\varphi$ be a homeomorphism of $K$ into itself, and $w \in C(K)$. We consider the weighted composition operator $T = wT_\varphi$ on $C(K)$ defined as
\begin{equation}\label{eq1}
  (Tf)(k) = w(k)f(\varphi(k)), \; f \in C(K), \; k \in K.
\end{equation}

By the reasons outlined in the introduction we will always assume that
\begin{equation}\label{eq2}
  \varphi(K) \subsetneqq K.
\end{equation}

We have to introduce some additional notations.

\begin{equation}\label{eq3}
  L = \bigcap \limits_{n=0}^\infty \varphi^n(K), \; M = K \setminus Int_K L, \; N = L \setminus Int_K L.
\end{equation}

\noindent Obviously, $\varphi$ is a homeomorphism of $L$ and $N$ onto themselves and~(\ref{eq1}) defines the action of $T$ on the spaces $C(L)$, $C(M)$, and $C(N)$.

\begin{theorem} \label{t1} Let $K$ be a compact Hausdorff space, $\varphi$ be a homeomorphism of $K$ into itself, and $w \in C(K)$. Let $T$ be the operator on $C(K)$ defined by~(\ref{eq1}). Assume~(\ref{eq2}) and notations in~(\ref{eq3}). Then
\begin{enumerate}[(I)]
  \item $\sigma(T, C(M))$ is either the disk $\rho(T, C(M))\mathds{U}$ or the singleton $\{0\}$.
  \item $\sigma(T) = \sigma(T, C(M)) \cup \sigma(T, C(L))$.
 \end{enumerate}
 \end{theorem}

 \begin{proof} (I) follows from~(\ref{eq2}) and Theorems 3.10 and 3.12 in~\cite{Ki1}.

 The proof of (II) will be divided into several steps.

 \noindent Step 1. We will prove the inclusion $\sigma(T, C(M)) \subseteqq \sigma(T)$. Assume to the contrary that there is a $\lambda \in \mathds{C}$, $\lambda \in \sigma(T, C(M)) \setminus \sigma(T)$. Because $0 \in \sigma(T)$, we can assume without loss of generality that $\lambda = 1$. Then $(I-T)C(K) = C(K)$ and because $\varphi(Int_K L) = Int_K L = \varphi^{-1}(Int_K L)$ we also have $(I-T)C(M) = C(M)$. Because $1 \in \sigma(T,C(M))$ there is an $f \in
C(M)$ such that $f \neq 0$ and $Tf = f$.  Then it follows from Lemma 3.6  in ~\cite{Ki1} that there is a point $k \in M$ such that
\begin{equation}\label{eq4}
   |w_n(k)| \geq 1, \; |w_n(\varphi^{-n})| \leq 1, n \in \mathds{N}.
\end{equation}
The point $k$ is either not $\varphi$-periodic or, in virtue of~(\ref{eq2}), a limit point of the set of all non $\varphi$-periodic points in $K$. It follows from the proof of Theorem 3.7 in~\cite{Ki1} that $\mathds{T} \subset \sigma(T)$, in contradiction with our assumption.

\noindent Step 2. On this step we prove the inclusion $\sigma(T, C(L)) \subseteqq \sigma(T)$.
  Let $\lambda \in \sigma(T,C(L) \setminus \sigma(T)$. We can assume that $\lambda = 1$, and like on the previous step $(I-T)C(K) = C(K)$ implies that $(I-T)C(L) = C(L)$. Therefore there is an $f \in C(L)$, $f \neq 0$, such that $Tf = f$. Consider two possibilities.

(a) $f \not \equiv 0$ on $L \setminus Int_K L$. Let $k \in L \setminus Int_K L$ be such that $|f(k)| = \max \limits_{L \setminus Int_K L} |f|$. Then like on step 1 we see that $\mathds{T} \subseteq \sigma(T)$.

(b) $f \equiv 0$ on $L \setminus Int_K L$. We will define the function $\tilde{f} \in C(K)$ as
\begin{equation*}
\tilde{f}(k) = \begin{cases} f(k) &\mbox{if } k \in L \\
0 & \mbox{if } k \in K \setminus L \end{cases} .
\end{equation*}
Then $T\tilde{f} = \tilde{f}$, and $1 \in \sigma(T)$ contrary to our assumption.

  Combining steps 1 and 2 we see that $ \sigma(T, C(M)) \cup \sigma(T, C(L)) \subseteqq \sigma(T)$.

  \noindent Step 3. We prove the inclusion $ \sigma(T) \subseteqq \sigma(T, C(M)) \cup \sigma(T, C(L))$.
   Let $\lambda \in \sigma(T)$. If $\lambda = 0$ then $\lambda \in D$ and therefore without loss of generality we can assume that $\lambda = 1$.

Consider first the case when $1 \in \sigma_{ap}(T)$. Then there is a sequence $f_n \in C(K)$, $\|f_n\|=1$ and $f_n \mathop \rightarrow \limits_{n \to \infty} 0$. But then clearly either $\|f_n\|_{C(L)} \not \rightarrow 0$ or $\|f_n\|_{C(M)} \not \rightarrow 0$, and therefore
\begin{equation*}
 1 \in \sigma_{ap}(T, C(L) \cup \sigma_{ap}(T,C(M)) \subseteq D \cup \sigma(T, C(L)).
\end{equation*}
If on the other hand $1 \in \sigma_r(T,C(K))$ then there is a regular nonzero Borel measure $\mu$ on $K$, $\mu \in C(K)^\prime$, such that $T^\prime \mu = \mu$. It is easy to see that $supp(\mu) \subseteq L$ whence $1 \in \sigma(T,C(L))$.
 \end{proof}

 \section{Some axillary results}

 To obtain a description of the essential spectra of $T$ we will need a series of lemmas. In the statements of all of the lemmas we will assume, without mentioning it explicitly, that $T$ is an operator on $C(K)$ defined by~(\ref{eq1}), that $\varphi$ is a homeomorphism of $K$ into itself, and that~(\ref{eq2}) holds. We will also assume notations from~(\ref{eq3}).

 \begin{lemma} \label{l1} Assume that $T$ is invertible on $C(L)$ and that $0 <|\lambda| < 1/\rho(T^{-1}, C(L))$. Then $(\lambda I - T)C(K) = C(K)$.
\end{lemma}

\begin{proof} It is enough to prove that the operator $\lambda I - T^\prime$ is bounded from below, where $T^\prime$ is the Banach dual of $T$. Assume to the contrary that there is a sequence $\mu_n \in (C(K))^\prime$ such that $\|\mu_n\| =1$ and $T^\prime \mu_n - \lambda \mu_n \rightarrow 0$. Because the operator $T^\prime$ preserves disjointness (see e.g.~\cite[Lemma 5.13]{Ki3}) we have $|T^\prime||\mu_n| - |\lambda||\mu_n| \rightarrow 0$. Let $\mu \in C(K)^\prime$ be a limit point of the set $\{|\mu_n|\}$ in the weak$^\star$ topology. Then $\mu$ is a probability measure on $K$. Because the operator $|T^\prime|= |T|^\prime$ is weak$^\star$ continuous we have $|T^\prime| \mu = |\lambda|\mu$. But then $supp( \mu) \subseteqq L$ whence $|\lambda| \in \sigma(|T|, C(L))$. The last statement involves a contradiction because the operator $|T|$ is invertible on $C(L)$ and $\rho(|T|^{-1},C(L)) = \rho(T^{-1},C(L))$.
\end{proof}

\begin{lemma} \label{l2}  (1) Let $\lambda \in \sigma_{ap}(T, C(N))$. Then $\lambda \mathds{T} \subseteq \sigma_2(T)$.

\noindent (2) Let $\lambda \in \sigma_{ap}(T^\prime, C^\prime(N))$. Then $\lambda \mathds{T} \subseteq \sigma_2(T^\prime)$.

\end{lemma}

\begin{proof} We divide the proof into four steps.

 (I) Let $\lambda = 0 \in \sigma_{ap}(T,C(N))$. Then the weight $w$ takes value $0$ on $N$. It follows from the definition of $N$ that there are pairwise distinct points $k_n \in K$ such that $|w(k_n)| \leq 1/n$. Let $u_n$ be the characteristic function of the singleton $\{k_n\}$. Then $u_n \in C^{\prime \prime}(K)$, $\|u_n\|=1$, the sequence $u_n$ is singular, and $T^{\prime \prime}u_n \rightarrow 0$. Thus $0 \in \sigma_2(T^{\prime \prime}) = \sigma_2(T)$.

(II) Let $0 \in \sigma_{ap}(T^\prime, C^\prime(N))$. Because $T^\prime = (T_\varphi)^\prime w^\prime$ and $(T_\varphi)^\prime$ is an isometry of $C^\prime(N)$ the weight $w$ takes value $0$ on $N$. Let $k_n$ be as in part (I) of the proof and $\delta_n$ be the Dirac measure corresponding to the point $k_n$. Then the sequence $\delta_n$ is singular and $T^\prime \delta_n \rightarrow 0$.

(III) Let $\lambda \in \sigma_{ap}(T,C(N))$ and $\lambda \neq 0$.  Without loss of generality we can assume that $|\lambda| = 1$. Recall that the restriction of $\varphi$ on $N$ is a homeomorphism of $N$ onto itself. Therefore by~\cite[Lemma 3.6]{Ki1} there is a point $k \in N$ such that $|w_n(k)| \geq 1$ and $|w_n(\varphi^{-n}(k))| \leq 1$, $n \in \mathds{N}$. Let us fix an $m \in \mathds{N}$. From the definition of the set $N$ follows that there is a net $\{k_\alpha\}$ of points in $K \setminus L$ convergent to $\varphi^{-m}(k)$. From this trivial observation and from the fact that $K \setminus L$ does not contain $\varphi$-periodic points easily follows the existence of points $k_n \in K \setminus L, n \in \mathds{N}$ with the properties.
\begin{enumerate}[(a)]
  \item The points $\varphi^i(k_n), -n-1 \leq i \leq n+1$ are pairwise distinct.
  \item The sets $A_n = \{\varphi^i(k_n), -n-1 \leq i \leq n+1\}$ are pairwise disjoint.
  \item For any $n \in \mathds{N}$ the following inequalities hold
   \begin{equation}\label{eq5}
     |w_i(k_n)| \geq 1/2 \; \text{and} \; |w_i(\varphi^{-i}(k_n))| \leq 2.
   \end{equation}
\end{enumerate}
Let $u_n$ be the characteristic function of the singleton $\{\varphi^n(k_n)\}$. Then $u_n \in C^{\prime \prime}(K)$. Let us fix $\alpha \in \mathds{C}$ such that $|\alpha| = 1$. Consider $F_n \in C^{\prime \prime}(K)$,
\begin{equation} \label{eq6}
 F_n = \sum \limits_{i=0}^{2n} \big{(} 1 - \frac{1}{\sqrt{n}}\big{)}^{|i-n|}
\alpha^{-i} (T^{\prime \prime})^i u_n.
\end{equation}
  It follows from~(\ref{eq5}) and~\ref{eq6} by the means of a simple estimate (see also~\cite[ Proof of Theorem 3.7]{Ki1}) that
 \begin{equation} \label{eq7}
 \|T^{\prime \prime} F_n - \alpha F_n \| = o(\|F_n\|), n \rightarrow \infty.
 \end{equation}
 Condition (b) guarantees that the sequence $F_n$ is singular and therefore~(\ref{eq7}) implies that
 \begin{equation*}
  \alpha \in \sigma_2(T^{\prime \prime}) = \sigma_2(T).
  \end{equation*}

 (IV) Let $\lambda \in \sigma_{ap}(T^\prime, C^\prime(N))$. It was proved in~\cite{Ki1} that there is a point $k \in N$ such that $|w_n(k)| \leq 1$ and $|w_n(\varphi^{-n}(k))| \geq 1$. Then we can find points $k_n \in K \setminus L$ satisfying conditions (a) and (b) above and also the following condition
 \begin{equation}\label{eq8}
   |w_i(k_n)| \leq 2 \; \text{and} \; |w_i(\varphi^{-i}(k_n))| \geq 1/2, \; n \in \mathds{N}.
 \end{equation}
  Let $\nu_n$ be the Dirac measure $\delta_{\varphi^{-n}(k_n)}$, $\alpha \in \mathds{T}$, and
  \begin{equation}\label{eq9}
     \mu_n = \sum \limits_{i=0}^{2n} \big{(} 1 - \frac{1}{\sqrt{n}}\big{)}^{|i-n|}
\alpha^{-i} (T^{\prime})^i \nu_n,
  \end{equation}
   It follows from~(\ref{eq8}) and~(\ref{eq9} that
 \begin{equation} \label{eq10}
   \|T^{\prime} \mu_n - \alpha \mu_n \| = o(\|\mu_n\|), n \rightarrow \infty
   \end{equation}
 Condition (b) guarantees that the sequence $\mu_n$ is singular, and therefore~(\ref{eq10}) implies that $\alpha \in \sigma_2(T^\prime)$.
 \end{proof}

  \begin{lemma} \label{l5} $\sigma_2(T, C(L)) \subseteqq \sigma_2(T)$ and
$\sigma_2(T^\prime, C^\prime(L)) \subseteqq \sigma_2(T^\prime)$.
 \end{lemma}

 \begin{proof} Let $\lambda \in \sigma_2(T, C(L))$. Then there is a singular sequence $f_n \in C(L)$ such that $\|f_n\|=1$ and $Tf_n - \lambda f_n \rightarrow 0$. We have to consider two possibilities.

 \noindent (1) $\|f_n\|_{C(N)} \not \rightarrow 0$. Then $\lambda \in \sigma_2(T)$ by Lemma~\ref{l2} (1).

 \noindent (2) $\|f_n\|_{C(N)} \rightarrow 0$. Then we can find $g_n \in C(L)$ such that $f_n - g_n \rightarrow 0$ and $g_n \equiv 0$ on $N$. Clearly, the sequence $g_n$ is singular in $C(L)$. We define the function $h_n \in C(K)$ as follows
 \begin{equation*}
h_n(k) = \begin{cases} g_n(k) &\mbox{if } k \in L \\
0 & \mbox{if } k \in K \setminus L \end{cases} .
\end{equation*}
  The sequence $h_n$ is singular in $C(K)$ and $Th_n - \lambda h_n \rightarrow 0$. Therefore $\lambda \in \sigma_2(T)$.

 The second inclusion is trivial.
 \end{proof}

 \begin{lemma} \label{l7} Let $|\lambda| > \rho(T,C(N))$ and $\lambda \not \in \sigma_2(T,C(L))$. Then $\lambda \not \in \sigma_2(T)$.
 \end{lemma}

 \begin{proof} Assume to the contrary that there is a singular sequence $f_n \in C(K)$ such that $\|f_n\| = 1$ and $Tf_n - \lambda f_n \rightarrow 0$. Because $|\lambda| > \rho(T, C(N))$ and $\rho(T, C(M)) = \rho(T, C(N))$ ( see e.g.~\cite[Theorem 3.23]{Ki1}), we have $\|f_n\|_{C(M)} \rightarrow 0$. Therefore, if $g_n$ is the restriction of $f_n$ on $L$ then the sequence $g_n$ is singular in $C(L)$, $\|g_n\| \rightarrow 1$, and $Tg_n - \lambda g_n \rightarrow 0$. Thus, $\lambda \in \sigma_2(T,C(L))$, a contradiction.
 \end{proof}

 \begin{lemma} \label{l8} Let $T$ be invertible on $C(N)$ and $|\lambda| < 1/\rho(T^{-1},C(N))$. Assume also that $\lambda \not \in \sigma_2(T, C(L))$. Then the following statements are equivalent.

 \noindent (1) $\lambda \in \sigma_2(T)$.

 \noindent (2) $card(K \setminus \varphi(K)) = \infty$.
 \end{lemma}

 \begin{proof} By Theorem~\ref{t1} we have $\lambda \in \sigma(T)$ and by Lemma~\ref{l1} $(\lambda I - T)C(K) = C(K)$. Therefore $\lambda \not \in \sigma_2(T)$ if and only if  $dim \; ker(\lambda I - T) = dim \; ker(\lambda I - T^{\prime \prime}) < \infty$.

       Assume that $card(K \setminus \varphi(K) < \infty$. This condition combined with $\lambda \not \in \sigma(T, C(N))$ provides that $dim \; \ker{((\lambda I - T), C(M))} < \infty$. Combining it with the condition $\lambda \not \in \sigma_2(T, C(L))$ we see that $dim \; ker(\lambda I - T) < \infty$.
       Thus, $(1) \Rightarrow (2)$.

       Assume next that $card(K \setminus \varphi(K)) = \infty$. Then clearly $dim \; ker(\lambda I - T) = \infty$ and therefore $\lambda \in \sigma_2(T)$.
  \end{proof}

 \begin{lemma} \label{l9} The set $\sigma_2(T, C(M))$ is rotation invariant and for a $\lambda \in \sigma(T,C(M))$, $\lambda \neq 0$, the following conditions are equivalent.

 \noindent (1) $\lambda \mathds{T} \cap \sigma_2(T,C(M)) = \emptyset $.

 \noindent (2) $M$ is the union of two clopen (in $M$) subsets $M_1$ and $M_2$ such that
 \begin{enumerate}[(a)]
   \item $M_2 \neq \emptyset$,
   \item $\varphi(M_i) \subseteqq M_i, i=1,2$,
   \item  If $M_1 \neq \emptyset$ then $\rho(T,C(M_1)) < |\lambda|$,
   \item $T$ is invertible on $C(N_2)$ and $|\lambda| < 1/\rho(T^{-1},C(N_2))$ where $N_2 = \bigcap \limits_{n=0}^\infty \varphi^n(M_2)$,
   \item $card(M_2 \setminus \varphi(M_2)) < \infty$.
 \end{enumerate}

 \end{lemma}

 \begin{proof} The implication $(2) \Rightarrow (1)$ follows from Lemmas~\ref{l7} and~\ref{l8}.

 To prove that $(1) \Rightarrow (2)$ notice that if $\lambda \in \sigma(T, C(M)) \setminus \sigma_2(T,C(M))$ then by Lemma~\ref{l2} we have $\lambda \mathds{T} \cap \sigma_{ap}(T, C(N)) = \emptyset$. We have to consider two possibilities.

 (I) $\lambda \mathds{T} \cap \sigma(T,C(N)) = \emptyset$. Then (see~\cite{Ki1}) $N$ is the union of two clopen (in $N$) subsets $N_1$ and $N_2$ (one of them might be empty), such that
 \begin{equation*}
   \varphi(N_i) = N_i, i=1,2,
 \end{equation*}
 \begin{equation*}
   \rho(T, C(N_1)) < |\lambda|\ ,
 \end{equation*}
 \begin{equation*}
    T \; \text{is invertible on } C(N_2)\; \text{and} \;|\lambda| < 1/\rho(T^{-1},C(N_2)).
 \end{equation*}
 It follows from the definition of $N$ that $M$ is the union of two clopen (in $M$) subsets $M_1$ and $M_2$ such that $N_i = \bigcap \limits_{n=0}^\infty \varphi^n(M_i), i=1,2$. It remains to apply Lemmas~\ref{l7} and~\ref{l8}.

 (II) $\lambda \mathds{T} \subset \sigma_r(T,C(N))$. Then (see~\cite[Theorem 3.29]{Ki1}) $N$ is the union of three pairwise disjoint nonempty subsets $N_1$, $N_2$, and $O$ such that

 \noindent $(\alpha)$ $N_i$, $i=1,2$ are closed subsets of $N$,

 \noindent $(\beta)$ $\varphi(N_i) = N_i$, i=1,2,

\noindent $(\gamma)$ $\rho(T,C(N_1)) < |\lambda|$,

\noindent $(\delta)$ The operator $T$ is invertible on $C(N_2)$ and $|\lambda| < 1/\rho(T,C(N_2))$,

\noindent $(\varepsilon)$ If $V_1$ and $V_2$ are open neighborhoods in $N$ of $N_1$ and $N_2$, respectively, then there is an $n \in \mathds{N}$, such that for any $m \geq n$ we have $\varphi^m(N \setminus V_2) \subseteq V_1$.

We need to consider two subcases.

\noindent $(IIa)$ For any open (in $M$) neighborhood $V$ of $N_2$ there is an infinite subset $E$ of $M \setminus \varphi(M)$ such that
\begin{equation*}
  \forall k \in E \; \exists n = n(k) \in \mathds{N} \; \text{such that} \; \varphi^n(k) \in V.
\end{equation*}
It follows from $(\delta)$ that there are a positive number $\varepsilon$ and open (in $M$) neighborhoods $V_n, n \in \mathds{N}$ of $N_2$ such that
\begin{equation}\label{eq12}
  |w_n(t)| \geq (|\lambda| + \varepsilon)^n, t \in V_n.
\end{equation}
By our assumption there are pairwise distinct points $k_n, n \in \mathds{N}$ and positive integers $m_n$ such that $k_n \in M \setminus \varphi(M)$ and $u_n = \varphi^{m_n}(k_n) \in V_n$. We define $f_n \in C^{\prime \prime}(M)$ as follows.
\begin{equation*}
  f_n(u_n)=1,
\end{equation*}
\begin{equation*}
  f_n(\varphi^{-l}(u_n)) =\frac{w_l(\varphi^{-l}(u_n)}{\lambda^l}, l=1, \ldots , m_n,
\end{equation*}
\begin{equation*}
   f_n(\varphi^l(u_n)) = \frac{\lambda^l}{w_l(u_n)}, l = 1, \ldots , n,
\end{equation*}
\begin{equation*}
   f_n(k) =0 \; \text{otherwise}.
\end{equation*}
It follows from the definition of $f_n$ and~(\ref{eq12})  that $\|f_n\| \geq 1$ and
$T^{\prime \prime} f_n - \lambda f_n \rightarrow 0$. Because the sequence $f_n$ is singular we get $\lambda \in \sigma_2(T^{\prime \prime}, C^{\prime \prime}(M))= \sigma_2(T, C(M))$, a contradiction.

\noindent $(IIb)$ There is an open (in $M$) neighborhood $V$ of $N_2$ such that the set
\begin{equation*}
  F = \{k \in M \setminus \varphi(M) \; : \;\exists n \in \mathds{N} \; \text{such that} \; \varphi^n(k) \in V\}
\end{equation*}
is at most finite. It follows from the definition of $N$ that $F$ cannot be empty. Clearly $F$ consists of points isolated in $M$. We will bring the assumption that $F$ is finite to a contradiction.
It is not difficult to see from $(\varepsilon)$ that there is a $k \in F$ such that the intersection of $cl\{\varphi^n(k) : n \in \mathds{N}\}$ with each of the sets $N_1$, $N_2$, and $O$ is not empty. Therefore we can assume without loss of generality that $M = cl\{\varphi^n(k) : n \in \mathds{N}\}$.

Let $W$ be an open neighborhood of $N_1$ in $M$ such that $cl W \cap N_2 = \emptyset$. It follows from $(\varepsilon)$ that there is an $m \in \mathds{N}$ such that $\varphi^m(W) \subseteq W$. Considering, if necessary, the operator $T^m$ instead of $T$ we can assume that $m = 1$. There is a $p \in \mathds{N}$ such that $\varphi^p(k) \in W$. Then $\varphi^n(k) \in W$ for any $n \geq p$, a contradiction.
 \end{proof}

\begin{lemma} \label{l10}
\begin{equation*}
   \sigma_2(T^\prime , C^\prime(M)) \cup \{0\}
 = \sigma_2(T^\prime , C^\prime(N)) \mathds{T} \cup \{0\}.
\end{equation*}
\end{lemma}

\begin{proof} The inclusion $\sigma_2(T^\prime , C^\prime(N)) \mathds{T} \cup \{0\} \subseteq \sigma_2(T^\prime, C^\prime(M)) \cup \{0\}$ follows from Lemma~\ref{l2}.

 To prove the converse inclusion consider $\lambda \in \sigma_{ap}(T', C'(N)) \setminus \{0\}$. The proof of Lemma~\ref{l1} shows that $|\lambda| \in \sigma_{ap}(|T'|, C'(N))$. But then (see~\cite{Ki3}) $\lambda \in \sigma_{ap}(T', C'(N)) \mathds{T}$.
\end{proof}

\section{Description of essential spectra of $T = wT_\varphi$}

Finally we can provide a complete description of essential spectra of weighted composition operators on $C(K)$ induced by non-surjective homeomorphisms. The statements of Theorems~\ref{t2} and~\ref{t3} below follow from the previous lemmas.

\begin{theorem} \label{t2}  Let $K$ be a compact Hausdorff space, $\varphi$ be a homeomorphism of $K$ into itself, and $w \in C(K)$. Let $T$ be the operator on $C(K)$ defined by~(\ref{eq1}). Assume~(\ref{eq2}) and notations in~(\ref{eq3}). Let $\lambda \in \sigma(T) \setminus \{0\}$. The operator $\lambda I - T$ is upper semi-Fredholm if and only if the following conditions are satisfied

\begin{enumerate}[(a)]
  \item The operator $\lambda I - T$ is upper semi-Fredholm on $C(L)$.
  \item The set $M$ is the union of two $\varphi$-invariant disjoint closed subsets $M_1$ and $M_2$ such that
  \item if $M_1 \neq \emptyset$ then $\rho(T,C(M_1)) < |\lambda|$,
  \item if $M_2 \neq \emptyset$ \footnote{In particular, if $\lambda \not \in \sigma(T,C(L))$.} then $T$ is invertible on $C(N_2)$, where $N_2 = \bigcap \limits_{n=0}^\infty \varphi^n(M_2)$, $|\lambda| < 1/\rho(T^{-1},C(N_2))$, and
   the set $M_2 \setminus \varphi(M_2)$ is finite.
\end{enumerate}
Moreover,
\begin{equation*}
  \dim{\ker{(\lambda I - T)}} = \dim{\ker{(\lambda I - T, C(L))}} + card(M_2 \setminus \varphi(M_2)).
\end{equation*}
  \end{theorem}

\begin{theorem} \label{t3} Let $K$ be a compact Hausdorff space, $\varphi$ be a homeomorphism of $K$ into itself, and $w \in C(K)$. Let $T$ be the operator on $C(K)$ defined by~(\ref{eq1}). Assume~(\ref{eq2}) and notations in~(\ref{eq3}). Let $\lambda \in \sigma(T) \setminus \{0\}$. The operator $\lambda I - T$ is lower semi-Fredholm if and only if the following conditions are satisfied
\begin{enumerate} [(a)]
  \item  The operator $\lambda I - T$ is lower semi-Fredholm on $C(L)$.
\item $\lambda \mathds{T} \subseteq \sigma_r(T', C'(N))$.
\end{enumerate}
Moreover, $def(\lambda I - T) = def(\lambda I - T), C(L))$.
\end{theorem}

\begin{corollary} \label{c1} Assume conditions of Theorem~\ref{t2}. Let $\lambda \in \sigma(T) \setminus \{0\}$. The operator $\lambda I - T$ is Fredholm if and only if it is Fredholm on $C(L)$ and conditions (b) - (d) from the statement of Theorem~\ref{t2} are satisfied.

Moreover $ind(\lambda I - T) = ind(\lambda I - T, C(L) + card(M_2 \setminus \varphi(M_2)$.
\end{corollary}

\begin{corollary} \label{c3} Assume conditions of Theorem~\ref{t2}. Assume additionally that the set of all $\varphi$-periodic points is of first category in $K$. Then the spectrum $\sigma(T)$ and the essential spectra $\sigma_i(T), i = 1, \cdots , 5$ are rotation invariant.
\end{corollary}

\begin{corollary} \label{c4} Assume conditions of Theorem~\ref{t1}.
\begin{enumerate}
  \item If the set of all isolated $\varphi$-periodic points is empty, then $\sigma_5(T) = \sigma(T)$.
  \item If $K$ has no isolated points (in particular, if $Int_K L = \emptyset$), then $\sigma_3(T) = \sigma(T)$
\end{enumerate}
  \end{corollary}

\begin{proof} The proof follows from Theorems~\ref{t1} and~\ref{t2}, and from~\cite[Theorems 2.7 and 2.11]{Ki3}.
\end{proof}

To finish our description of essential spectra of $T$ it remains to look at the case $\lambda = 0$;

\begin{theorem} \label{t4} Assume conditions of Theorem~\ref{t1}. Then

\noindent $(1)$ The operator $T$ is upper semi-Fredholm if and only if the following two conditions are satisfied
\begin{enumerate}[(a)]
  \item The set $Z(w) = \{k \in K : w(k) = 0\}$ is either empty or all of its points are isolated in $K$,
  \item the set $K \setminus \varphi(K)$ is finite.
\end{enumerate}
   Moreover, $\dim{\ker{ T}} = card ((K \setminus \varphi(K)) \cup Z(w))$.

\noindent $(2)$ The operator $T$ is lower semi-Fredholm if and only if the set $Z(w) = \{k \in K : w(k) = 0\}$ is either empty or all of its points are isolated in $K$.

\noindent Moreover, $def \, T = card \; Z(w)$.

\noindent $(3)$ The operator $T$ is Fredholm if and only if it is semi-Fredholm and $\dim{\ker{T}} < \infty$.

\noindent $(4)$ The operator $T$ is Fredholm and $ind\, T =0$ if and only if $T$ is Fredholm and $w \equiv 0$ on $K \setminus \varphi(K)$.

\noindent $(5)$ $0  \in \sigma_5(T)$.
\end{theorem}

\begin{proof} $(1)$ Assume that $T$ is semi-Fredholm and that $\dim{\ker{T}} < \infty$. Then the same is true for $T^{\prime \prime}$. If $k \in K \setminus \varphi(K)$ then $T^{\prime \prime} \chi_k = 0$ where $\chi_k \in C(K)^{\prime \prime}$ is the characteristic function of the singleton $\{k\}$. Therefore $card(K \setminus \varphi(K) < \infty$.

Similarly, if $k \in \varphi(K)$ and $w(k) = 0$ then $T^{\prime \prime} \chi_{\varphi(k)} = 0$ whence $Z(w)$ is finite or empty. Assume now that $w(k) = 0$ but $k$ is not isolated in $K$. Then there is a sequence of pairwise distinct points $k_n \in K$ such that $|w(k_n)| \leq 1/n$. The sequence $\chi_{\varphi(k_n)}$ is singular in $C(K)^{\prime \prime}$ and $T^{\prime \prime}\chi_{\varphi(k_n)} \rightarrow 0$ whence $0 \in \sigma_2(T)$.

Conversely, assume conditions $(a)$ and $(b)$. Assume also, contrary to the statement of the theorem that there is a singular sequence $f_n \in C(K)$ such that $\|f_n\| = 1$ and $Tf_n \rightarrow 0$. It is immediate to see that $f_n \rightarrow 0$ uniformly on $E = \varphi(K) \setminus \varphi(Z(w))$. Because the set $K \setminus E$ is finite the sequence $f_n$ contains a convergent subsequence and thus cannot be singular.

Finally, if $Tf =0$ then $supp \; f \subseteq K \setminus \varphi(K)) \cup Z(w)$ whence $\dim{\ker{ T}} = card ((K \setminus \varphi(K)) \cup Z(w))$.

\noindent $(2)$ Assume that $T$ is semi-Fredholm and that $def T < \infty$. If $k \in Z(w)$ then $T^\prime \delta_k = 0$ whence $Z(w)$ is either finite or empty.

Conversely, if $card \; Z(w) < \infty$, $\|\mu_n\| = 1$, and $T^\prime \mu_n \rightarrow 0$ then (because $T^\prime$ preserves disjointness) $|T^\prime| |\mu_n| = |T^\prime \mu_n| \rightarrow 0$. Let $\nu_{1n}$ and $\nu_{2n}$ be the restrictions of the measure $|\mu_n|$ on $Z(w)$ and $K \setminus Z(w)$, respectively. Because $Z(w)$ is finite there is a positive constant $c$ such that $|w| > c$ on $K \setminus Z(w)$. Therefore $\nu_{2n} \rightarrow 0$ and we can find a norm convergent subsequence of the sequence $\mu_n$. Therefore, $0 \not \in \sigma_2(T^\prime)$.

\noindent It is immediate to see that if $T^\prime \mu = 0$ then $supp \; \mu \subseteq Z(w)$ whence $def \, T = card \; Z(w)$.

\noindent $(3)$ and $(4)$ follow immediately from $(1)$ and $(2)$.

\noindent $(5)$ If $\sigma(T, C(M))$ is a disk of positive radius then it follows directly from the definition of $\sigma_5(T)$ that $0 \in \sigma_5(T)$. On the other hand, if $\rho(T, C(M)) = 0$ then there is a point $k \in N$ such that $w(k) = 0$. Because $k$ is not an isolated point of $K$ we see that $0 \in \sigma_2(T) \cap \sigma_2(T^\prime) = \sigma_1(T) \subseteq \sigma_5(T)$.
\end{proof}

\begin{example} \label{e1} Let $Tf(x) = f(x/2), \; f \in C[0,1], \; x \in [0,1]$. Then

\noindent (1) $\sigma_5(T) = \sigma_4(T) = \sigma_3(T) = \sigma_2(T) = \sigma(T) = \mathds{D}$.

\noindent (2) $\sigma_2(T^\prime) = \sigma_1(T) = \mathds{T}$.
\end{example}

\end{document}